\documentclass[11pt]{amsart}
\usepackage{bbm}
\usepackage{mathrsfs}
\usepackage{amsfonts}
\usepackage{amssymb}

\newtheorem{theorem}{Theorem}[section]

\newtheorem{cor}[theorem]{Corollary}
\newtheorem{lemma}[theorem]{Lemma}
\newtheorem{definition}[theorem]{Definition}
\newtheorem{remark}[theorem]{Remark}
\newtheorem{example}[theorem]{Example}
\newtheorem{conjecture}[theorem]{Conjecture}
\newtheorem{problem}[theorem]{Problem}

\def\I{\mathcal{I}}
\def\Z{{\mathbb{Z}}}

\def\N{{\mathbb{N}}}

\def\C{{\mathbb{C}}}

\def\Y{{\mathbb{Y}}}

\def\bh{{\mathbf{h}}}
\def\bf{{\mathbf{f}}}

\def\bg{{\mathbf{g}}}

\def\lc{\hbox{\rm{lc}}}
\def\lt{\hbox{\rm{lt}}}

\def\D{{\sigma}}

\def\mb{{\mathbbm{m}}}
\def\st{\hbox{ {\rm{s.t.} }}}

\def\sat{\hbox{\rm{sat}}}

\begin{document}
\title{Finite Basis for Radical Well-Mixed Difference Ideals Generated by Binomials}
\author{Jie Wang}
\address{KLMM, Academy of Mathematics and Systems Science, The Chinese Academy of Sciences, Beijing 100190, China}
\email{wangjie212@mails.ucas.ac.cn}
\date{\today}

\begin{abstract}
In this paper, we prove a finite basis theorem for radical well-mixed difference ideals generated by binomials. As a consequence, every strictly ascending chain of radical well-mixed difference ideals generated by binomials in a difference polynomial ring is finite, which solves an open problem of difference algebra raised by E. Hrushovski in the binomial case.
\end{abstract}

\maketitle
\bibliographystyle{amsplain}

\section{Introduction}
In \cite{Hrushovski1}, E. Hrushovski developed the theory of difference scheme, which is one of the major recent advances in difference algebra geometry. In Hrushovski's theory, well-mixed difference ideals played an important role. So it is significant to make clear of the properties of well-mixed difference ideals.

It is well known that Hilbert's basis theorem does not hold for difference ideals in a difference polynomial ring. Instead, we have Ritt-Raudenbush basis theorem which asserts that every perfect difference ideal in a difference polynomial ring has a finite basis. It is naturally to ask if the finitely generated property holds for more difference ideals. Let $K$ be a difference field and $R$ a finitely difference generated difference algebra over $K$. In \cite[Section 4.6]{Hrushovski1}, Hrushovski raised the problem whether a radical well-mixed difference ideal in $R$ is finitely generated. The problem is also equivalent to whether the ascending chain condition holds for radical well-mixed difference ideals in $R$. For the sake of convenience, let us state it as a conjecture:
\begin{conjecture}\label{intro-conj}
Every strictly ascending chain of radical well-mixed difference ideals in $R$ is finite.
\end{conjecture}

Also in \cite[Section 4.6]{Hrushovski1}, Hrushovski proved that the answer is yes under some additional assumptions on $R$. In \cite{levin}, A. Levin showed that the ascending chain condition does not hold if we drop the radical condition. The counter example given by Levin is a well-mixed difference ideal generated by binomials. In \cite[Section 9]{wibmer-group}, M. Wibmer showed that if $R$ can be equipped with the structure of a difference Hopf algebra over $K$, then Conjecture \ref{intro-conj} is valid. In \cite{Jie}, J. Wang proved that Conjecture \ref{intro-conj} is valid for radical well-mixed difference ideals generated by monomials.

Difference ideals generated by binomials were first studied by X.~S. Gao, Z. Huang, C.~M. Yuan in \cite{gao-dbi}. Some basic properties of difference ideals generated by binomials were proved in that paper due to the correspondence between $\Z[x]$-lattices and normal binomial difference ideals.

The main result of this paper is that every radical well-mixed difference ideal generated by binomials in a difference polynomial ring is finitely generated. As a consequence, Conjucture \ref{intro-conj} is valid for radical well-mixed difference ideals generated by binomials in a difference polynomial ring.


\section{Preliminaries}
\subsection{Preliminaries for Difference Algebra}
We recall some basic notions from difference algebra. Standard references are \cite{levin,wibmer}. All rings in this paper will be assumed to be commutative and unital.

A {\em difference ring}, or {\em $\sigma$-ring} for short, is a ring $R$ together with a ring endomorphism $\sigma\colon R\rightarrow R$, and we call $\D$ a {\em difference operator} on $R$. If $R$ is a field, then we call it a {\em difference field}, or {\em $\sigma$-field} for short.
In this paper, all $\sigma$-fields will be assumed to be of characteristic $0$.

Following \cite{gao-tdv}, we introduce the following notation of symbolic exponents. Let $x$ be an algebraic indeterminate and $p=\sum_{i=0}^s c_i x^i\in\N[x]$. For $a$ in a $\sigma$-ring, we denote $a^p = \prod_{i=0}^s (\sigma^i(a))^{c_i}$ with $\sigma^0(a)=a$ and $a^0=1$. It is easy to check that for $p,q\in\N[x], a^{p+q}=a^{p} a^{q}, a^{pq}=(a^{p})^{q}$.

Let $R$ be a $\D$-ring. A {\em $\sigma$-ideal} $I$ in $R$ is an algebraic ideal which is closed under $\sigma$, i.e., $\sigma(I)\subseteq I$. If $I$ also has the property that $a^x\in I$ implies $a\in I$, it is called a {\em reflexive $\sigma$-ideal}. A {\em $\sigma$-prime} ideal is a reflexive $\sigma$-ideal which is prime as an algebraic ideal. A $\sigma$-ideal $I$ is said to be {\em well-mixed} if for $a,b\in K\{\Y\}$, $ab\in I$ implies $ab^x\in I$. A $\sigma$-ideal $I$ is said to be {\em perfect} if for $g\in\N[x]\setminus\{0\}$ and $a\in K\{\Y\}$, $a^g\in I$ implies $a\in I$. It is easy to prove that every perfect $\sigma$-ideal is well-mixed and every $\sigma$-prime ideal is perfect.

If $F\subseteq R$ is a subset of $R$, denote the minimal ideal containing $F$ by $(F)$, the minimal $\D$-ideal containing $F$ by $[F]$ and denote the minimal well-mixed $\D$-ideal, the minimal radical well-mixed $\D$-ideal, the minimal perfect $\D$-ideal containing $F$ by $\langle F\rangle$, $\langle F\rangle_r$, $\{F\}$ respectively, which are called the {\em well-mixed closure}, the {\em radical well-mixed closure}, the {\em perfect closure} of $F$ respectively.

Let $K$ be a $\sigma$-field and $\Y=(y_1,\ldots,y_n)$ a tuple of $\sigma$-indeterminates over $K$. Then the {\em $\sigma$-polynomial ring} over $K$ in $\Y$ is the polynomial ring in the variables $y_i^{x^j}$
for $j\in\N$ and $i=1,\ldots,n$. It is denoted by
$K\{\Y\}=K\{y_1,\ldots,y_n\}$
and has a natural $K$-$\sigma$-algebra structure.

\subsection{Preliminaries for Binomial Difference Ideals}
A {\em $\Z[x]$-lattice} is a $\Z[x]$-submodule of $\Z[x]^n$ for some $n$. Since $\Z[x]^n$ is Noetherian as a $\Z[x]$-module, we see that any $\Z[x]$-lattice is finitely generated as a $\Z[x]$-module.

Let $K$ be a $\sigma$-field and $\Y=(y_1,\ldots,y_n)$ a tuple of $\sigma$-indeterminates over $K$. For $\bf=(f_1,\ldots,f_n)\in\N[x]^n$, we define $\Y^{\bf}=\prod_{i=1}^ny_i^{f_i}$. $\Y^{\bf}$ is called a {\em monomial} in $\Y$ and $\bf$ is called its {\em support}. For $a,b\in K^*$ and $\bf,\bg\in\N[x]^n$, $a\Y^{\bf}+b\Y^{\bg}$ is called a {\em binomial}. If $a=1,b=-1$, then $\Y^{\bf}-\Y^{\bg}$ is called a {\em pure binomial}. A {\em (pure) binomial $\D$-ideal} is a $\D$-ideal generated by (pure) binomials.

For $f\in\Z[x]$, we write $f=f_+-f_-$, where $f_+,f_-\in\N[x]$ are the positive part and the negative part of $f$ respectively. For $\bf\in\Z[x]^n$, $\bf_+=(f_{1_+},\ldots,f_{n_+})$, $\bf_-=(f_{1_-},\ldots,f_{n_-})$.

\begin{definition}
A {\em partial character} $\rho$ on a $\Z[x]$-lattice $L$ is a group homomorphism from $L$ to the multiplicative group $K^*$ satisfying $\rho(x\bf)=(\rho(\bf))^x$ for all $\bf\in L$.
\end{definition}
A {\em trivial} partial character on $L$ is defined by setting $\rho(\bf)=1$ for all $\bf\in L$.

Given a partial character $\rho$ on a $\Z[x]$-lattice $L$, we define the following binomial $\sigma$-ideal in $K\{\Y\}$,
\begin{equation*}
\I_L(\rho):=[\Y^{\bf_{+}}-\rho(\bf)\Y^{\bf_{-}}\mid\bf\in L].
\end{equation*}
$L$ is called the {\em support lattice} of $\I_L(\rho)$. In particular, if $\rho$ is a trivial partial character on $L$, then the binomial $\sigma$-ideal defined by $\rho$ is called a {\em lattice $\D$-ideal}, which is denoted by $\I_L$.


Let $\mb$ be the multiplicatively closed set generated by $y_i^{x^j}$ for $i=1,\ldots,n,j\in\N$. A $\sigma$-ideal $I$ is said to be {\em normal} if for any $M\in\mb$ and $p\in K\{\Y\}$, $Mp\in I$ implies $p\in I$. For any $\sigma$-ideal $I$,
$$I:\mb=\{p\in K\{\Y\}\mid\exists M\in\mb \st Mp\in I\}$$
is a normal $\sigma$-ideal.
\begin{lemma}
A normal binomial $\D$-ideal is radical.
\end{lemma}
\begin{proof}
For the proof, please refer to \cite{gao-dbi}.
\end{proof}

In \cite{gao-dbi}, it was proved that there is a one-to-one correspondence between normal binomial $\sigma$-ideals and partial characters $\rho$ on some $\Z[x]$-lattice $L$.


In \cite{gao-dbi}, the concept of {\em M-saturation} of a $\Z[x]$-lattice was introduced.
\begin{definition}
Assume $K$ is algebraically closed. If a $\Z[x]$-lattice $L$ satisfies
\begin{equation}\label{eq-per}
m\bf\in L\Rightarrow (x-o_m)\bf\in L,
\end{equation}
where $m\in\N$, $\bf\in\Z[x]^n$, and $o_m$ is the $m$-th transforming degree of the unity of $K$, then it is said to be {\em M-saturated}. For any $\Z[x]$-lattice $L$, the smallest M-saturated $\Z[x]$-lattice containing $L$ is called the {\em M-saturation} of $L$ and is denoted by $\sat_{M}(L)$.
\end{definition}
The following two lemmas were proved in \cite{gao-dbi} for the Laurent case and it is easy to generalize to the normal case.
\begin{lemma}\label{lm2}
Assume $K$ is algebraically closed and inversive. Let $\rho$ be a partial character on a $\Z[x]$-lattice $L$. If $\I_L(\rho)$ is well-mixed, then $L$ is M-saturated. Conversely, if $L$ is M-saturated, then either $\langle\I_L(\rho)\rangle:\mb=[1]$ or $\I_L(\rho)$ is well-mixed.
\end{lemma}

\begin{lemma}\label{lm0}
Let $K$ be an algebraically closed and inversive $\sigma$-field and $\rho$ a partial character on a $\Z[x]$-lattice $L$. Then $\langle\I_L(\rho)\rangle_r:\mb$ is either $[1]$ or a normal binomial $\sigma$-ideal
whose support lattice is $\sat_M(L)$. In particular, $\langle\I_L\rangle_r:\mb$ is either $[1]$ or $\I_{\sat_M(L)}$.
\end{lemma}

\section{Radical Well-Mixed Difference Ideal Generated by Binomials is Fininitely Generated}
In this section, we will prove every radical well-mixed $\D$-ideal generated by binomials in a $\D$-polynomial ring is finitely generated as a radical well-mixed $\D$-ideal. For simplicity, we only consider the case for pure binomials since it is easy to generalize the results to any binomials.

For convenience, for $h\in\Z[x]$, if $\deg(h_+)>\deg(h_-)$, we denote $h^+=h_+$ and $h^-=h_-$. Otherwise, we denote $h^+=h_-$ and $h^-=h_+$. Moveover, we set $\deg(0)=-1$.

For $a,b,c,d\in\N$, we define $ax^b>cx^d$ if $b>d$, or $b=d$ and $a>c$. For $h\in\Z[x]$, we use $\lt(h)$ and $\lc(h)$ to denote the leading term of $h$ and the leading coefficient of $h$ respectively.
\begin{theorem}\label{thm1}
For any $\Z[x]$-lattice $L\subseteq\Z[x]^n$, $\langle \I_L\rangle_r$ is finitely generated as a radical well-mixed $\D$-ideal.
\end{theorem}
\begin{proof}
Denote the set of all maps from $\{1,\ldots,n\}$ to $\{+,-,0\}$ by $\Lambda$ and $\tau_0\in\Lambda$ is the map such that $\tau_0(i)=0$ for $1\le i\le n$. Let $\Lambda_0=\Lambda\backslash\{\tau_0\}$. For any $\tau\in\Lambda_0$, define
\begin{align*}
A_{\tau}:=&\{(h_1,\ldots,h_n)\in L\mid \lc(h_i)>0 \textrm{ if } \tau(i)=+, \textrm{ } \lc(h_i)<0 \textrm{ if } \tau(i)=-,
\textrm{ and } \lc(h_i)=0\\ &\textrm{ if } \tau(i)=0, \textrm{ } i=1,\ldots,n\},
\end{align*}
and
\begin{align*}
\Sigma_{\tau}:=\{(\deg(h_1^+),\lc(h_1^+),\ldots,\deg(h_n^+),\lc(h_n^+),\deg(h_1^-),\ldots,\deg(h_n^-))\mid (h_1,\ldots,h_n)\in A_{\tau}\}.
\end{align*}
For any $\tau\in\Lambda_0$, let $G_{\tau}$ be the subset of $A_{\tau}$ such that
\begin{align*}
\{(\deg(h_1^+),\lc(h_1^+),\ldots,\deg(h_n^+),\lc(h_n^+),\deg(g_1^-),\ldots,\deg(g_n^-))\mid \bg=(g_1,\ldots,g_n)\in G_{\tau}\}
\end{align*}
is the set of minimal elements in $\Sigma_{\tau}$ under the product order.
It follows that $G_{\tau}$ is a finite set. Let
$$F_{\tau}:=\{\Y^{\bg_+}-\Y^{\bg_-}\mid \bg\in G_{\tau}\}.$$
We claim that the finite set $\cup_{\tau\in\Lambda_0}F_{\tau}$ generates $\langle \I_L\rangle_r$ as a radical well-mixed $\D$-ideal.

Denote $\I_0=\langle\cup_{\tau\in\Lambda_0}F_{\tau}\rangle_r$. We will prove the claim by showing that $\Y^{\bh_+}-\Y^{\bh_-}\in\I_0$ for all $\bh\in L$. Let us do induction on $(\lt(h_1^+),\ldots,\lt(h_n^+))$ under the lexicographic order for $\bh=(h_1,\ldots,h_n)\in L$. For the simplicity, we will assume that $\Y^{\bh_+}-\Y^{\bh_-}$ has the form
$$y_1^{h_1^+}\cdots y_t^{h_t^+}y_{t+1}^{h_{t+1}^-}\cdots y_n^{h_n^{-}}-y_1^{h_1^-}\cdots y_t^{h_t^-}y_{t+1}^{h_{t+1}^+}\cdots y_n^{h_n^{+}},$$
where $1\le t\le n$. And without loss of generality, we furthermore assume $\lc(h_i)\ne0$ for $1\le i\le n$.

The case for $\bh=\mathbf{0}$ is trivial. Now for the inductive step. By definition, there exists $\tau\in\Lambda_0$ and $(g_1,\ldots,g_n)\in G_{\tau}$ such that $(h_1,\ldots,h_n)\in A_{\tau}$ and $\deg(g_i^+)\le\deg(h_i^+),\lc(g_i^+)\le\lc(h_i^+)$,$\deg(g_i^-)\le\deg(h_i^-), i=1,\ldots,n$. Choose $j\in\{1,\ldots,n\}$ such that
$$\deg(h_j^+)-\deg(g_j^+)=\min_{1\le i\le n}\{\deg(h_i^+)-\deg(g_i^+)\}.$$
Without loss of generality, we can assume $j=1$. Let $s=\deg(h_1^+)-\deg(g_1^+)\ge0$. Since $\lc(h_1^+)\ge\lc(g_1^+)$, there exists an $e\in\N[x]$ such that $\deg(e)<\deg(h_1^+)$ and $p:=h_1^++e-x^sg_1^+\in\N[x]$ with $\lt(p)<\lt(h_1^+)$. Then
\begin{align*}
&y_1^{e}y_2^{x^sg_2^+}\cdots y_{t}^{x^sg_t^+}y_{t+1}^{x^sg_{t+1}^-}\cdots y_n^{x^sg_n^-}(y_1^{h_1^+}\cdots y_t^{h_t^+}y_{t+1}^{h_{t+1}^-}\cdots y_n^{h_n^{-}}-y_1^{h_1^-}\cdots y_t^{h_t^-}y_{t+1}^{h_{t+1}^+}\cdots y_n^{h_n^{+}})\\
={}&y_1^{p+x^sg_1^+}y_2^{h_2^++x^sg_2^+}\cdots y_t^{h_t^++x^sg_t^+}y_{t+1}^{h_{t+1}^-+x^sg_{t+1}^-}\cdots y_n^{h_n^-+x^sg_n^-}\\
-{}&y_1^{h_1^-+e}y_2^{h_2^-+x^sg_2^+}\cdots y_t^{h_t^-+x^sg_t^+}y_{t+1}^{h_{t+1}^++x^sg_{t+1}^-}\cdots y_n^{h_n^++x^sg_n^-}\\
={}&(y_1^{g_1^+}\cdots y_t^{g_t^+}y_{t+1}^{g_{t+1}^-}\cdots y_n^{g_n^{-}}-y_1^{g_1^-}\cdots y_t^{g_t^-}y_{t+1}^{g_{t+1}^+}\cdots y_n^{g_n^{+}})^{x^s}y_1^{p}y_2^{h_2^+}\cdots y_t^{h_t^+}y_{t+1}^{h_{t+1}^-}\cdots y_n^{h_n^-}\\
+{}&y_1^{p+x^sg_1^-}y_2^{h_2^++x^sg_2^-}\cdots y_t^{h_t^++x^sg_t^-}y_{t+1}^{h_{t+1}^-+x^sg_{t+1}^+}\cdots y_n^{h_n^-+x^sg_n^+}\\
-{}&y_1^{h_1^-+e}y_2^{h_2^-+x^sg_2^+}\cdots y_t^{h_t^-+x^sg_t^+}y_{t+1}^{h_{t+1}^++x^sg_{t+1}^-}\cdots y_n^{h_n^++x^sg_n^-}\\
={}&(y_1^{g_1^+}\cdots y_t^{g_t^+}y_{t+1}^{g_{t+1}^-}\cdots y_n^{g_n^{-}}-y_1^{g_1^-}\cdots y_t^{g_t^-}y_{t+1}^{g_{t+1}^+}\cdots y_n^{g_n^{+}})^{x^s}y_1^{p}y_2^{h_2^+}\cdots y_t^{h_t^+}y_{t+1}^{h_{t+1}^-}\cdots y_n^{h_n^-}\\
+{}&y_1^{d_1}\cdots y_n^{d_n}(y_1^{w_{1_+}}\cdots y_n^{w_{n_{+}}}-y_1^{w_{1_-}}\cdots y_n^{w_{n_-}}).
\end{align*}
Since $\lt(p+x^sg_1^-)<\lt(h_1^+),\lt(h_1^-+e)<\lt(h_1^+)$, then $\lt(w_1^+)<\lt(h_1^+)$ and because of the choice of $j$, we have $s+\deg(g_i^+)\le\deg(h_i^+)$ for $2\le i\le n$, from which it follows $\lt(w_i^+)\le\lt(h_i^+), 2\le i\le n$. Therefore, $(\lt(w_1^+),\ldots,\lt(w_n^+))<(\lt(h_1^+),\ldots,\lt(h_n^+))$. Thus by the induction hypothesis, $y_1^{w_{1_+}}\cdots y_n^{w_{n_{+}}}-y_1^{w_{1_-}}\cdots y_n^{w_{n_-}}\in\I_0$ and hence
$$y_1^{e}y_2^{x^sg_2^+}\cdots y_{t}^{x^sg_t^+}y_{t+1}^{x^sg_{t+1}^-}\cdots y_n^{x^sg_n^-}(\Y^{\bh_+}-\Y^{\bh_-})\in\I_0.$$
So by the properties of radical well-mixed $\D$-ideals, we have
$$y_1^{x^sg_1^+}\cdots y_{t}^{x^sg_t^+}y_{t+1}^{x^sg_{t+1}^-}\cdots y_n^{x^sg_n^-}(\Y^{\bh_+}-\Y^{\bh_-})\in\I_0,$$
and then
$$y_1^{x^sg_1^-}\cdots y_{t}^{x^sg_t^-}y_{t+1}^{x^sg_{t+1}^+}\cdots y_n^{x^sg_n^+}(\Y^{\bh_+}-\Y^{\bh_-})\in\I_0.$$
If $s>0$, let $s'=\max\{0,s-\min_{1\le i\le t}\{\deg(g_i^+)-\deg(g_i^-)\}\}<s$. Again by the properties of radical well-mixed $\D$-ideals, we have
$$y_1^{x^{s'}g_1^+}\cdots y_{t}^{x^{s'}g_t^+}y_{t+1}^{x^sg_{t+1}^+}\cdots y_n^{x^sg_n^+}(\Y^{\bh_+}-\Y^{\bh_-})\in\I_0,$$
and then
$$y_1^{x^{s'}g_1^-}\cdots y_{t}^{x^{s'}g_t^-}y_{t+1}^{x^sg_{t+1}^+}\cdots y_n^{x^sg_n^+}(\Y^{\bh_+}-\Y^{\bh_-})\in\I_0.$$
If $s'>0$, repeat the above process, and we eventually obtain
$$y_1^{g_1^-}\cdots y_{t}^{g_t^-}y_{t+1}^{x^sg_{t+1}^+}\cdots y_n^{x^sg_n^+}(\Y^{\bh_+}-\Y^{\bh_-})\in\I_0.$$
Since $\deg(g_i^-)\le\deg(h_i^-), 1\le i\le t$ and $s+\deg(g_i^+)\le\deg(h_i^+), t+1\le i\le n$, then by the properties of radical well-mixed $\D$-ideals, we have
\begin{equation}\label{eq5}
\Y^{\bh_-}(\Y^{\bh_+}-\Y^{\bh_-})\in \I_0.
\end{equation}

Similarly, we also have
\begin{equation}\label{eq6}
\Y^{\bh_+}(\Y^{\bh_+}-\Y^{\bh_-})\in\I_0.
\end{equation}
Combining (\ref{eq5}) and (\ref{eq6}), we obtain $(\Y^{\bh_+}-\Y^{\bh_-})^2\in \I_0$, and hence $\Y^{\bh_+}-\Y^{\bh_-}\in\I_0$.
So we complete the proof.
\end{proof}

\begin{cor}\label{thm2}
Let $L\subseteq\Z[x]^n$ be a $\Z[x]$-lattice such that $\I_L$ is well-mixed, then $\I_L$ is finitely generated as a radical well-mixed $\D$-ideal.
\end{cor}
\begin{proof}
It is immediate from Theorem \ref{thm1} since $\I_L$ is already a radical well-mixed $\D$-ideal.
\end{proof}

\begin{example}\label{ex1}
Let $L=(\begin{pmatrix}x-1\\1-x\end{pmatrix})\subseteq\Z[x]^2$ be a $\Z[x]$-lattice. $\I_L$ is a $\D$-prime $\D$-ideal.
Then $\I_L=[y_1^{x^{i}}y_2-y_1y_2^{x^{i}}:i\in\N^*]=\langle y_1^xy_2-y_1y_2^x\rangle_r$.
\end{example}

\begin{example}\label{ex2}
Let $L=(\begin{pmatrix}x^2+1-x\\x-x^2-1\end{pmatrix})\subseteq\Z[x]^2$ be a $\Z[x]$-lattice. $\I_L$ is a $\D$-prime $\D$-ideal.
Then $\I_L=\langle y_1^{x^2+1}y_2^x-y_1^xy_2^{x^2+1},y_1^{x^3+1}-y_2^{x^3+1}\rangle_r$.
\end{example}

\begin{example}\label{ex4}
Let $L=(\begin{pmatrix}x^2+1-x\\x-1\end{pmatrix})\subseteq\Z[x]^2$ be a $\Z[x]$-lattice. $\I_L$ is a $\D$-prime $\D$-ideal.
Then $\I_L=\langle y_1^{x^2+1}y_2^x-y_1^xy_2,y_1^{x^3+1}y_2^{x^2}-y_2\rangle_r$.
\end{example}

To show radical well-mixed $\D$-ideals generated by any binomials are finitely generated, we need the following lemma.
\begin{lemma}[\cite{Jie}, Proposition 5.2]\label{lm6}
Let $F$ and $G$ be subsets of any $\D$-ring $R$. Then
\[\langle F\rangle_r\cap\langle G\rangle_r=\langle FG\rangle_r.\]
As a corollary, if $I$ and $J$ are two $\D$-ideals of $R$, then
\[\langle I\rangle_r\cap\langle J\rangle_r=\langle I\cap J\rangle_r=\langle IJ\rangle_r.\]
\end{lemma}
\begin{proof}
For the proof, please refer to \cite{Jie}.
\end{proof}

\begin{lemma}\label{lm8}
Suppose $I\subseteq K\{y_1,\ldots,y_n\}$ is a pure binomial $\D$-ideal. Then $\langle I\rangle_r:\mb$ is finitely generated as a radical well-mixed $\D$-ideal.
\end{lemma}
\begin{proof}
Since $I:\mb$ is a normal binomial $\D$-ideal, there exists a $\Z[x]$-lattice $L$ such that $I:\mb=\I_L$. Note that $\langle I\rangle_r:\mb$=$\langle I:\mb\rangle_r:\mb$, so by Lemma \ref{lm0}, $\langle I\rangle_r:\mb$ is $[1]$ or $\I_{\sat_M(L)}$. Since $\langle I\rangle_r$ is radical well-mixed, it is easy to show that $\langle I\rangle_r:\mb$ is radical well-mixed. So by Corollary \ref{thm2}, $\langle I\rangle_r:\mb$ is finitely generated as a radical well-mixed $\D$-ideal.
\end{proof}

\begin{lemma}\label{lm1}
Suppose $I\subseteq K\{y_1,\ldots,y_n\}$ is a pure binomial $\D$-ideal. Then
$$\langle I\rangle_r=\langle I\rangle_r:\mb\cap\langle I,y_{p_1}^{x^{a_1}}\rangle_r\cap\cdots\cap\langle I,y_{p_l}^{x^{a_l}}\rangle_r$$
for some $\{p_1,\ldots,p_l\}\subseteq\{1,\ldots,n\}$ and some $(a_1,\ldots,a_l)\in\N^l$.
\end{lemma}
\begin{proof}
By Lemma \ref{lm8}, $\langle I\rangle_r:\mb$ is finitely generated as a radical well-mixed $\D$-ideal. Therefore, there exist $f_1,\ldots,f_s\in\langle I\rangle_r:\mb$ and $m_1,\ldots,m_s\in\mb$ such that $\langle I\rangle_r:\mb=\langle f_1,\ldots,f_s\rangle_r$ and $m_1f_1,\ldots,m_sf_s\in\langle I\rangle_r$. Then by Lemma \ref{lm6},
\begin{align*}
\langle I\rangle_r&=\langle I,f_1\rangle_r\cap\langle I,m_1\rangle_r\\
&=\langle I,f_1,f_2\rangle_r\cap\langle I,f_1,m_2\rangle_r\cap\langle I,m_1\rangle_r\\
&=\langle I,f_1,f_2\rangle_r\cap\langle I,m_1m_2\rangle_r\\
&=\cdots\\
&=\langle f_1,\ldots,f_s\rangle_r\cap\langle I,m_1\cdots m_s\rangle_r\\
&=\langle I\rangle_r:\mb\cap\langle I,y_{p_1}^{x^{a_1}}\rangle_r\cap\cdots\cap\langle I,y_{p_l}^{x^{a_l}}\rangle_r,
\end{align*}
for some $\{p_1,\ldots,p_l\}\subseteq\{1,\ldots,n\}$ and some $(a_1,\ldots,a_l)\in\N^l$.
\end{proof}

Suppose $\{j_1,\ldots,j_t\}\subseteq\{1,\ldots,n\}$, $(a_1,\ldots,a_t)\in\N^t$ and $I_0\subseteq K\{y_1,\ldots,y_n\}$ is a pure binomial $\D$-ideal. Denote $T_{j_1\ldots j_t}^{a_1\ldots a_t}:=\{y_1^{f_1}\cdots y_{n}^{f_n}:f_1,\ldots,f_n\in\N[x],\deg(f_{j_i})<a_i,1\le i\le t\}$. We say $I_0$ is {\em saturated} with respect to $\{y_{j_1}^{x^{a_1}},\ldots,y_{j_t}^{x^{a_t}}\}$ if $I_0=I_0:T_{j_1\ldots j_t}^{a_1\ldots a_t}$, that is, for $g\in K\{y_1,\ldots,y_n\}$ and $M\in T_{j_1\ldots j_t}^{a_1\ldots a_t}$, $Mg\in I_0$ implies $g\in I_0$. Let $I\subseteq K\{y_1,\ldots,y_n\}$ be a pure binomial $\D$-ideal. The minimal $\D$-ideal containing $I$ which is saturated with respect to $\{y_{j_1}^{x^{a_1}},\ldots,y_{j_t}^{x^{a_t}}\}$ is called the {\em $T_{j_1\ldots j_t}^{a_1\ldots a_t}$-saturated closure} of $I$, denoted by $N_{j_1\ldots j_t}^{a_1\ldots a_t}(I)$. We will give a concrete description of the $T_{j_1\ldots j_t}^{a_1\ldots a_t}$-saturated closure of a pure binomial $\D$-ideal $I$. Let $I^{[0]}=I$ and recursively define $I^{[i]}=[I^{[i-1]}:T_{j_1\ldots j_t}^{a_1\ldots a_t}](i=1,2,\ldots)$. The following lemma is easy to check by definition.
\begin{lemma}\label{lm9}
Let $I\subseteq K\{y_1,\ldots,y_n\}$ be a pure binomial $\D$-ideal. Then
\begin{equation}\label{eq1}
N_{j_1\ldots j_t}^{a_1\ldots a_t}(I)=\cup_{i=0}^{\infty}I^{[i]}.
\end{equation}
\end{lemma}

Let $I_0\subseteq K\{y_1,\ldots,y_n\}$ be a pure binomial $\D$-ideal. Then we say $I=\langle I_0,y_{j_1}^{x^{a_1}},\ldots,y_{j_t}^{x^{a_t}}\rangle_r$ is {\em quasi-normal} if $I_0$ is saturated with respect to $\{y_{j_1}^{x^{a_1}},\ldots,y_{j_t}^{x^{a_t}}\}$ and for any binomial $\Y^{\bf}-\Y^{\bg}\in I_0$, if $\Y^{\bf}\in[y_{j_1}^{x^{a_1}},\ldots,y_{j_t}^{x^{a_t}}]$, then $\Y^{\bg}\in[y_{j_1}^{x^{a_1}},\ldots,y_{j_t}^{x^{a_t}}]$.
Similarly to Theorem \ref{thm1}, we can prove the following useful lemma.
\begin{lemma}\label{lm5}
Let $\{j_1,\ldots,j_t\}\subseteq\{1,\ldots,n\}$, $(a_1,\ldots,a_t)\in\N^t$ and $I_0\subseteq K\{y_1,\ldots,y_n\}$ a pure binomial $\D$-ideal. Assume that $I=\langle I_0,y_{j_1}^{x^{a_1}},\ldots,y_{j_t}^{x^{a_t}}\rangle_r$ is quasi-normal. Then $I$ is finitely generated as a radical well-mixed $\D$-ideal.
\end{lemma}
\begin{proof}
Let $J=\{\Y^{\bh_+}-\Y^{\bh_-}\in I_0\mid \Y^{\bh_+},\Y^{\bh_-}\in T_{j_1\ldots j_t}^{a_1\ldots a_t}\}$. Similarly to Theorem \ref{thm1}, we can prove $\langle J\rangle_r$ is finitely generated as a radical well-mixed $\D$-ideal. Thus $I=\langle J,y_{j_1}^{x^{a_1}},\ldots,y_{j_t}^{x^{a_t}}\rangle_r$ is finitely generated as a radical well-mixed $\D$-ideal.
\end{proof}

\begin{lemma}\label{lm3}
Suppose $\{j_1,\ldots,j_t\}\subseteq\{1,\ldots,n\}$, $(a_1,\ldots,a_t)\in\N^t$ and $I\subseteq K\{y_1,\ldots,y_n\}$ is a pure binomial $\D$-ideal. Let $I_0=N_{j_1\ldots j_t}^{a_1\ldots a_t}(I)$. Assume that $\langle I_0, y_{j_1}^{x^{a_1}},\ldots,y_{j_t}^{x^{a_t}}\rangle_r$ is quasi-normal. 
Then
$$\langle I,y_{j_1}^{x^{a_1}},\ldots,y_{j_t}^{x^{a_t}}\rangle_r=\langle I_0, y_{j_1}^{x^{a_1}},\ldots,y_{j_t}^{x^{a_t}}\rangle_r\cap\bigcap_{1\le k\le l}\langle I,y_{j_1}^{x^{a_1}},\ldots,y_{j_t}^{x^{a_t}},y_{p_k}^{x^{b_k}}\rangle_r,$$
where either $p_k\notin\{j_1,\ldots,j_t\}$, or $p_k=j_m$ and $b_k<a_m$ for $1\le k\le l$.
\end{lemma}
\begin{proof}
Since $\langle I_0,y_{j_1}^{x^{a_1}},\ldots,y_{j_t}^{x^{a_t}}\rangle_r$ is quasi-normal, by Lemma \ref{lm5}, it is finitely generated as a radical well-mixed $\D$-ideal. That is to say, there exist $f_1,\ldots,f_s\in I_0$ such that
$$\langle I_0,y_{j_1}^{x^{a_1}},\ldots,y_{j_t}^{x^{a_t}}\rangle_r=\langle f_1,\ldots,f_s,y_{j_1}^{x^{a_1}},\ldots,y_{j_t}^{x^{a_t}}\rangle_r.$$
By (\ref{eq1}), $I_0=\cup_{i=0}^{\infty}I^{[i]}$, so there exists $i\in\N$ such that $f_1,\ldots,f_s\in I^{[i]}$. By definition, there exist $g_{i1},\ldots,g_{il_i}\in I^{[i-1]}:T_{j_1\ldots j_t}^{a_1\ldots a_t}$ and $m_{i1},\ldots,m_{il_i}\in T_{j_1\ldots j_t}^{a_1\ldots a_t}$ such that $f_1,\ldots,f_s\in[g_{i1},\ldots,g_{il_i}]$ and $m_{i1}g_{i1},\ldots,m_{il_i}g_{il_i}\in I^{[i-1]}$. Again there exist $g_{i-11},\ldots,g_{i-1l_{i-1}}\in I^{[i-2]}:T_{j_1\ldots j_t}^{a_1\ldots a_t}$ and $m_{i-11},\ldots,m_{i-1l_{i-1}}\in T_{j_1\ldots j_t}^{a_1\ldots a_t}$ such that $m_{i1}g_{i1},\ldots,m_{il_i}g_{il_i}\in[g_{i-11},\ldots,g_{i-1l_{i-1}}]$ and $m_{i-11}g_{i-11},\ldots,m_{i-1l_{i-1}}g_{i-1l_{i-1}}\in I^{[i-2]}$. Iterating this process, we eventually have there exist $g_{11},\ldots,g_{1l_{1}}\in I:T_{j_1\ldots j_t}^{a_1\ldots a_t}$ and $m_{11},\ldots,m_{1l_{1}}\in T_{j_1\ldots j_t}^{a_1\ldots a_t}$ such that $m_{21}g_{21},\ldots,m_{2l_2}g_{2l_2}\in[g_{11},\ldots,g_{1l_{1}}]$ and $m_{11}g_{11},\ldots,m_{1l_{1}}g_{1l_{1}}\in I$.
So by Lemma \ref{lm6}, we obtain
\begin{align*}
\langle I,y_{j_1}^{x^{a_1}},\ldots,y_{j_t}^{x^{a_t}}\rangle_r&=\langle I,g_{11},\ldots,g_{1l_{1}},y_{j_1}^{x^{a_1}},\ldots,y_{j_t}^{x^{a_t}}\rangle_r\cap\langle I,m_{11}\cdots m_{1l_{1}},y_{j_1}^{x^{a_1}},\ldots,y_{j_t}^{x^{a_t}}\rangle_r\\
&=\langle I,g_{21},\ldots,g_{2l_{2}},g_{11},\ldots,g_{1l_{1}},y_{j_1}^{x^{a_1}},\ldots,y_{j_t}^{x^{a_t}}\rangle_r\\
&{\hspace{1em}}\cap\langle I,m_{21}\cdots m_{2l_{2}}m_{11}\cdots m_{1l_{1}},y_{j_1}^{x^{a_1}},\ldots,y_{j_t}^{x^{a_t}}\rangle_r\\
&=\cdots\\
&=\langle I,g_{i1},\ldots,g_{il_i},\ldots,g_{11},\ldots,g_{1l_{1}},y_{j_1}^{x^{a_1}},\ldots,y_{j_t}^{x^{a_t}}\rangle_r\\
&{\hspace{1em}}\cap\langle I,m_{i1}\cdots m_{il_{i}}\cdots m_{11}\cdots m_{1l_{1}},y_{j_1}^{x^{a_1}},\ldots,y_{j_t}^{x^{a_t}}\rangle_r\\
&=\langle I_0,y_{j_1}^{x^{a_1}},\ldots,y_{j_t}^{x^{a_t}}\rangle_r\cap\bigcap_{1\le k\le l}\langle I,y_{j_1}^{x^{a_1}},\ldots,y_{j_t}^{x^{a_t}},y_{p_k}^{x^{b_k}}\rangle_r,
\end{align*}
where either $p_k\notin\{j_1,\ldots,j_t\}$, or $p_k=j_m$ and $b_k<a_m$ for $1\le k\le l$.
\end{proof}

From the proof of Lemma \ref{lm3}, we obtain the following lemma which will be used later.
\begin{lemma}\label{lm11}
Suppose $\{j_1,\ldots,j_t\}\subseteq\{1,\ldots,n\}$, $(a_1,\ldots,a_t)\in\N^t$ and $I\subseteq K\{y_1,\ldots,y_n\}$ is a pure binomial $\D$-ideal. Let $h\in N_{j_1\ldots j_t}^{a_1\ldots a_t}(I)$.
Then
$$\langle I,y_{j_1}^{x^{a_1}},\ldots,y_{j_t}^{x^{a_t}}\rangle_r=\langle I', y_{j_1}^{x^{a_1}},\ldots,y_{j_t}^{x^{a_t}}\rangle_r\cap\bigcap_{1\le k\le l}\langle I,y_{j_1}^{x^{a_1}},\ldots,y_{j_t}^{x^{a_t}},y_{p_k}^{x^{b_k}}\rangle_r,$$
where $I'\supseteq [I,h]$ and either $p_k\notin\{j_1,\ldots,j_t\}$, or $p_k=j_m$ and $b_k<a_m$ for $1\le k\le l$.
\end{lemma}

\begin{lemma}\label{lm7}
Suppose $\{j_1,\ldots,j_t\}\subseteq\{1,\ldots,n\}$, $(a_1,\ldots,a_t)\in\N^t$ and $I\subseteq K\{y_1,\ldots,y_n\}$ is a pure binomial $\D$-ideal. Assume that there exists a binomial $\Y^{\bf}-\Y^{\bg}\in I$ such that $\Y^{\bf}\in[y_{j_1}^{x^{a_1}},\ldots,y_{j_t}^{x^{a_t}}]$ and $\Y^{\bg}\notin[y_{j_1}^{x^{a_1}},\ldots,y_{j_t}^{x^{a_t}}]$. Then
$$\langle I,y_{j_1}^{x^{a_1}},\ldots,y_{j_t}^{x^{a_t}}\rangle_r=\bigcap_{1\le k\le l}\langle I,y_{j_1}^{x^{a_1}},\ldots,y_{j_t}^{x^{a_t}},y_{p_k}^{x^{b_k}}\rangle_r,$$
where either $p_k\notin\{j_1,\ldots,j_t\}$, or $p_k=j_m$ and $b_k<a_m$ for $1\le k\le l$.
\end{lemma}
\begin{proof}
Since there exists a binomial $\Y^{\bf}-\Y^{\bg}\in I$ such that $\Y^{\bf}\in[y_{j_1}^{x^{a_1}},\ldots,y_{j_t}^{x^{a_t}}]$ and $\Y^{\bg}\notin[y_{j_1}^{x^{a_1}},\ldots,y_{j_t}^{x^{a_t}}]$, then $\Y^{\bg}\in\langle I,y_{j_1}^{x^{a_1}},\ldots,y_{j_t}^{x^{a_t}}\rangle_r$. Therefore, by the properties of radical well-mixed $\D$-ideals, there exist $\{p_1,\ldots,p_l\}\subseteq\{1,\ldots,n\}$ and $(b_1,\ldots,b_l)\in\N^l$ satisfying either $p_k\notin\{j_1,\ldots,j_t\}$, or $p_k=j_m$ and $b_k<a_m$, for $1\le k\le l$ such that $y_{p_1}^{x^{b_1}}\cdots y_{p_l}^{x^{b_l}}\in\langle I,y_{j_1}^{x^{a_1}},\ldots,y_{j_t}^{x^{a_t}}\rangle_r$. Hence,
$$\langle I,y_{j_1}^{x^{a_1}},\ldots,y_{j_t}^{x^{a_t}}\rangle_r=\bigcap_{1\le k\le l}\langle I,y_{j_1}^{x^{a_1}},\ldots,y_{j_t}^{x^{a_t}},y_{p_k}^{x^{b_k}}\rangle_r.$$
\end{proof}


\begin{lemma}\label{lm4}
Let $i\in\{1,\ldots,n\}$ and $a\in\N$. Suppose $I\subseteq K\{y_1,\ldots,y_n\}$ is a pure binomial $\D$-ideal. Then
$$\langle I,y_i^{x^a}\rangle_r=\bigcap_{(j_1,\ldots,j_t),(b_{j_1},\ldots,b_{j_t})}\langle I_{j_1,\ldots,j_t}^{b_{j_1}\ldots b_{j_t}},y_{j_1}^{x^{b_{j_1}}},\ldots,y_{j_t}^{x^{b_{j_t}}}\rangle_r$$
is a finite intersection, $i\in\{j_1,\ldots,j_t\}$ and for each member in the intersection, either $I_{j_1,\ldots,j_t}^{b_{j_1}\ldots b_{j_t}}\subseteq[y_{j_1}^{x^{b_{j_1}}},\ldots,y_{j_t}^{x^{b_{j_t}}}]$, or $\langle I_{j_1,\ldots,j_t}^{b_{j_1}\ldots b_{j_t}},y_{j_1}^{x^{b_{j_1}}},\ldots,y_{j_t}^{x^{b_{j_t}}}\rangle_r$ is quasi-normal.
\end{lemma}
\begin{proof}

Using Lemma \ref{lm7} repeatedly if there exists a binomial $\Y^{\bf}-\Y^{\bg}\in I$ such that $\Y^{\bf}\in[y_i^{x^a}]$ and $\Y^{\bg}\notin[y_i^{x^a}]$, assume that we obtain a decomposition as follows:
\begin{equation}\label{eq2}
\langle I,y_i^{x^a}\rangle_r=\bigcap_{(j_1\ldots j_t),(c_{j_1}\ldots c_{j_t})}\langle I_{j_1,\ldots,j_t}^{c_{j_1}\ldots c_{j_t}},y_{j_1}^{x^{c_{j_1}}},\ldots,y_{j_t}^{x^{c_{j_t}}}\rangle_r.
\end{equation}

For each $I_{j_1,\ldots,j_t}^{c_{j_1}\ldots c_{j_t}}$, if $I_{j_1,\ldots,j_t}^{c_{j_1}\ldots c_{j_t}}\subseteq[y_{j_1}^{x^{c_{j_1}}},\ldots,y_{j_t}^{x^{c_{j_t}}}]$, then we have nothing to do. Otherwise, let $I_0=N_{j_1\ldots j_t}^{c_{j_1}\ldots c_{j_t}}(I_{j_1,\ldots,j_t}^{c_{j_1}\ldots c_{j_t}})$. If there exists a binomial $\Y^{\bf}-\Y^{\bg}\in I_0$ such that $\Y^{\bf}\in[y_{j_1}^{x^{c_{j_1}}},\ldots,y_{j_t}^{x^{c_{j_t}}}]$ and $\Y^{\bg}\notin[y_{j_1}^{x^{c_{j_1}}},\ldots,y_{j_t}^{x^{c_{j_t}}}]$. Then by Lemma \ref{lm11},
$$\langle I_{j_1,\ldots,j_t}^{c_{j_1}\ldots c_{j_t}},y_{j_1}^{x^{c_{j_1}}},\ldots,y_{j_t}^{x^{c_{j_t}}}\rangle_r=\langle I', y_{j_1}^{x^{c_{j_1}}},\ldots,y_{j_t}^{x^{c_{j_t}}}\rangle_r\cap\bigcap_{1\le k\le l}\langle I_{j_1,\ldots,j_t}^{c_{j_1}\ldots c_{j_t}},y_{j_1}^{x^{c_{j_1}}},\ldots,y_{j_t}^{x^{c_{j_t}}},y_{p_k}^{x^{d_k}}\rangle_r,$$
where $I'\supseteq [I_{j_1,\ldots,j_t}^{c_{j_1}\ldots c_{j_t}},\Y^{\bf}-\Y^{\bg}]$ and either $p_k\notin\{j_1,\ldots,j_t\}$, or $p_k=j_m$ and $d_k<c_{j_m}$ for $1\le k\le l$. Therefore, by Lemma \ref{lm7}, we have
$$\langle I', y_{j_1}^{x^{c_{j_1}}},\ldots,y_{j_t}^{x^{c_{j_t}}}\rangle_r=\bigcap_{1\le k\le l'}\langle I', y_{j_1}^{x^{c_{j_1}}},\ldots,y_{j_t}^{x^{c_{j_t}}},y_{s_k}^{x^{e_k}}\rangle_r,$$
where either $s_k\notin\{j_1,\ldots,j_t\}$, or $s_k=j_m$ and $e_k<c_{j_m}$ for $1\le k\le l'$. Thus
\begin{align}\label{eq3}
\langle I_{j_1,\ldots,j_t}^{c_{j_1}\ldots c_{j_t}},y_{j_1}^{x^{c_{j_1}}},\ldots,y_{j_t}^{x^{c_{j_t}}}\rangle_r=&\bigcap_{1\le k\le l'}\langle I', y_{j_1}^{x^{c_{j_1}}},\ldots,y_{j_t}^{x^{c_{j_t}}},y_{s_k}^{x^{e_k}}\rangle_r\cap\\
&\bigcap_{1\le k\le l}\langle I_{j_1,\ldots,j_t}^{c_{j_1}\ldots c_{j_t}},y_{j_1}^{x^{c_{j_1}}},\ldots,y_{j_t}^{x^{c_{j_t}}},y_{p_k}^{x^{d_k}}\rangle_r.\notag
\end{align}
For each member in the intersection (\ref{eq3}), repeat the above process. Since at each step, either the number of elements of $\{y_{j_1},\ldots,y_{j_t}\}$ strictly increase, or the vector $(c_{j_1},\ldots,c_{j_t})$ strictly decrease, then in finite steps we must obtain either $I_{j_1,\ldots,j_t}^{c_{j_1}\ldots c_{j_t}}\subseteq[y_{j_1}^{x^{c_{j_1}}},\ldots,y_{j_t}^{x^{c_{j_t}}}]$, or for any binomial $\Y^{\bf}-\Y^{\bg}\in I_0$, if $\Y^{\bf}\in[y_{j_1}^{x^{c_{j_1}}},\ldots,y_{j_t}^{x^{c_{j_t}}}]$, then $\Y^{\bg}\in[y_{j_1}^{x^{c_{j_1}}},\ldots,y_{j_t}^{x^{c_{j_t}}}]$. In the latter case, by Lemma \ref{lm3},
$$\langle I_{j_1,\ldots,j_t}^{c_{j_1}\ldots c_{j_t}},y_{j_1}^{x^{c_{j_1}}},\ldots,y_{j_t}^{x^{c_{j_t}}}\rangle_r=\langle I_0, y_{j_1}^{x^{c_{j_1}}},\ldots,y_{j_t}^{x^{c_{j_t}}}\rangle_r\cap\bigcap_{1\le k\le l''}\langle I_{j_1,\ldots,j_t}^{c_{j_1}\ldots c_{j_t}},y_{j_1}^{x^{c_{j_1}}},\ldots,y_{j_t}^{x^{c_{j_t}}},y_{t_k}^{x^{b_k}}\rangle_r,$$
where either $t_k\notin\{j_1,\ldots,j_t\}$, or $t_k=j_m$ and $b_k<c_{j_m}$ for $1\le k\le l''$.

Apply the same procedure to the rest of the members in the intersection, and in finite steps we obtain the desired decomposition.
\end{proof}

Now we can prove the main theorem of this paper.
\begin{theorem}\label{thm3}
Suppose $I\subseteq K\{y_1,\ldots,y_n\}$ is a pure binomial $\D$-ideal. Then $\langle I\rangle_r$ is finitely generated as a radical well-mixed $\D$-ideal.
\end{theorem}
\begin{proof}
By Lemma \ref{lm1}, we have
\begin{equation}\label{eq8}
\langle I\rangle_r=\langle I\rangle_r:\mb\cap\langle I,y_{p_1}^{x^{a_1}}\rangle_r\cap\cdots\cap\langle I,y_{p_l}^{x^{a_l}}\rangle_r
\end{equation}
for some $\{p_1,\ldots,p_l\}\subseteq\{1,\ldots,n\}$ and some $\{a_1,\ldots,a_l\}\in\N^l$.
By Lemma \ref{lm4},
\begin{align}\label{eq7}
\langle I,y_{p_k}^{x^{a_k}}\rangle_r=\bigcap_{(j_1\ldots j_t),(b_{j_1}\ldots b_{j_t})}\langle I_{j_1\ldots j_t}^{b_{j_1}\ldots b_{j_t}},y_{j_1}^{x^{b_{j_1}}},\ldots,y_{j_t}^{x^{b_{j_t}}}\rangle_r.
\end{align}
Since in (\ref{eq7}), either $I_{j_1,\ldots,j_t}^{b_{j_1}\ldots b_{j_t}}\subseteq[y_{j_1}^{x^{b_{j_1}}},\ldots,y_{j_t}^{x^{b_{j_t}}}]$, or $\langle I_{j_1,\ldots,j_t}^{b_{j_1}\ldots b_{j_t}},y_{j_1}^{x^{b_{j_1}}},\ldots,y_{j_t}^{x^{b_{j_t}}}\rangle_r$ is quasi-normal, then by Lemma \ref{lm5}, each member in the intersection (\ref{eq7}) is finitely generated as a radical well-mixed $\D$-ideal. And since (\ref{eq7}) is a finite intersection, by Lemma \ref{lm6}, $\langle I,y_{p_k}^{x^{a_k}}\rangle_r$ is finitely generated as a radical well-mixed $\D$-ideal for $1\le k\le l$. Moreover, by Lemma \ref{lm8}, $\langle I\rangle_r:\mb$ is finitely generated as a radical well-mixed $\D$-ideal. Putting all above together, by (\ref{eq8}) and Lemma \ref{lm6}, $\langle I\rangle_r$ is finitely generated as a radical well-mixed $\D$-ideal.
\end{proof}

\begin{cor}\label{cor}
Any strictly ascending chain of radical well-mixed $\D$-ideals generated by pure binomials in a $\D$-polynomial ring is finite.
\end{cor}
\begin{proof}
Assume that $I_1\subseteq I_2\subseteq\ldots\subseteq I_k\ldots$ is an ascending chain of radical well-mixed $\D$-ideals generated by pure binomials in a $\D$-polynomial ring. Then $\cup_{i=1}^{\infty}I_i$ is also a radical well-mixed $\D$-ideal generated by pure binomials. By Theorem \ref{thm3}, $\cup_{i=1}^{\infty}I_i$ is finitely generated as a radical well-mixed $\D$-ideal, say $\{a_1,\dots,a_m\}$. Then there exists $k\in \N$ large enough such that $\{a_1,\dots,a_m\}\subset I_k$. It follows $I_k=I_{k+1}=\ldots=\cup_{i=1}^{\infty}I_i$.
\end{proof}

\begin{remark}
By Corollary \ref{cor}, Conjecture \ref{intro-conj} is valid for radical well-mixed $\D$-ideals generated by pure binomials in a $\D$-polynomial ring.
\end{remark}

\begin{remark}
Theorem \ref{thm3} and Corollary \ref{cor} actually hold for radical well-mixed $\D$-ideals generated by any binomials (not necessarily pure binomials).
\end{remark}

In \cite{levin1}, A. Levin gave an example to show that a strictly ascending chain of well-mixed $\D$-ideals in a $\D$-polynomial ring may be infinite. Here we give a simpler example.
\begin{example}
Let $I=\langle y_1^xy_2-y_1y_2^x\rangle$ and $I_0=[y_1^xy_2-y_1y_2^x,y_1^{x^j}(y_1^{x^i}y_2-y_1y_2^{x^i})^{x^l},y_2^{x^j}(y_1^{x^i}y_2-y_1y_2^{x^i})^{x^l}:i,j,l\in\N,i\ge2,j\ge i-1]$. We claim that $I=I_0$. It is easy to check that $I_0\subseteq I$. So we only need to show that $I_0$ is already a well-mixed $\D$-ideal. Following Example \ref{ex1}, let $\I_L=\langle y_1^xy_2-y_1y_2^x\rangle_r$. Suppose $ab\in I_0\subseteq\I_L$. Since $\I_L=[y_1^{x^{i}}y_2-y_1y_2^{x^{i}}:i\in\N^*]$ is a $\D$-prime $\D$-ideal, then $a\in\I_L$ or $b\in\I_L$. In each case, we can easily deduce that $ab^x\in I_0$. Therefore, $I_0$ is well-mixed and $I=I_0$. Thus $y_1^{x^2}y_2-y_1y_2^{x^2}\notin I$. In fact, in a similar way we can show that $\langle y_1^xy_2-y_1y_2^x,\ldots,y_1^{x^k}y_2-y_1y_2^{x^k}\rangle=[y_1^xy_2-y_1y_2^x,\ldots,y_1^{x^k}y_2-y_1y_2^{x^k},
y_1^{x^j}(y_1^{x^i}y_2-y_1y_2^{x^i})^{x^l},y_2^{x^j}(y_1^{x^i}y_2-y_1y_2^{x^i})^{x^l}:i,j,l\in\N,i\ge k+1,j\ge i-k]$ and $y_1^{x^{k+1}}y_2-y_1y_2^{x^{k+1}}\notin\langle y_1^xy_2-y_1y_2^x,\ldots,y_1^{x^k}y_2-y_1y_2^{x^k}\rangle$ for $k\ge2$. So we obtain a strictly infinite ascending chain of well-mixed $\D$-ideals:
$$\langle y_1^xy_2-y_1y_2^x\rangle\subsetneq\langle y_1^xy_2-y_1y_2^x,y_1^{x^2}y_2-y_1y_2^{x^2}\rangle\subsetneq\cdots\subsetneq\langle y_1^xy_2-y_1y_2^x,\ldots,y_1^{x^k}y_2-y_1y_2^{x^k}\rangle\subsetneq\cdots.$$
As a consequence, $\I_L$ is not finitely generated as a well-mixed $\D$-ideal.
\end{example}


In \cite{gao-dbi}, it is shown that the radical closure, the reflexive closure, and the perfect closure of a binomial $\D$-ideal are still a binomial $\D$-ideal. However, the well-mixed closure of a binomial $\D$-ideal may not be a binomial $\D$-ideal. More precisely, it relies on the action of the difference operator. We will give an example to illustrate this. 
\begin{example}
Let $K=\C$ and $R=\C\{y_1,y_2,y_3,y_4\}$. Let us consider the $\D$-ideal $I=\langle y_1^2(y_3-y_4),y_2^2(y_3-y_4)\rangle$ of $R$. Since $(y_1^2-y_2^2)(y_3-y_4)=(y_1+y_2)(y_1-y_2)(y_3-y_4)\in I$, we have $(y_1+y_2)(y_1-y_2)^x(y_3-y_4)=(y_1^{x+1}+y_1^xy_2-y_1y_2^x-y_2^{x+1})(y_3-y_4)\in I$. Note that $y_1^{x+1}(y_3-y_4),y_2^{x+1}(y_3-y_4)\in I$, hence $(y_1^xy_2-y_1y_2^x)(y_3-y_4)\in I$. If the difference operator on $\C$ is the identity map, similarly to Example 4.1 in \cite{Jie}, we can show that $y_1^xy_2(y_3-y_4),y_1y_2^x(y_3-y_4)\notin I$. As a consequence, $I$ is not a binomial $\D$-ideal.


On the other hand, if the difference operator on $\C$ is the conjugation map(that is $\D(i)=-i$), the situation is totally changed. Since $(y_1^2+y_2^2)(y_3-y_4)=(y_1+iy_2)(y_1-iy_2)(y_3-y_4)\in I$, $(y_1+iy_2)(y_1-iy_2)^x(y_3-y_4)=(y_1^{x+1}+iy_1^xy_2+iy_1y_2^x-y_2^{x+1})(y_3-y_4)\in I$ and hence $(y_1^xy_2+y_1y_2^x)(y_3-y_4)\in I$. Since we also have $(y_1^xy_2-y_1y_2^x)(y_3-y_4)\in I$, then $y_1^xy_2(y_3-y_4),y_1y_2^x(y_3-y_4)\in I$. Actually $I=[y_1^u(y_3-y_4)^a,y_1^{w_1}y_2^{w_2}(y_3-y_4)^a,y_2^v(y_3-y_4)^a:u,v,w_1,w_2,a\in\N[x],2\preceq u,2\preceq v,x+1\preceq w_1+w_2]$ ($\preceq$ is defined in \cite{Jie}). In this case, $I=\langle y_1^2(y_3-y_4),y_2^2(y_3-y_4)\rangle$ is indeed a binomial $\D$-ideal.
\end{example}

\begin{problem}
We conjecture that the radical well-mixed closure of a binomial $\D$-ideal is still a binomial $\D$-ideal. However, we cannot prove it now.
\end{problem}

\textbf{Acknowledgement.} The author thanks Professor ChunMing Yuan for helpful discussions.


\begin{thebibliography}{99}
\bibitem{es-bi}
D. Eisenbud and B. Sturmfels.
Binomial Ideals,
{\em Duke Math. J.}, 84(1), 1-45, 1996.

\bibitem{gao-dbi}
X.~S. Gao, Z. Huang, C.~M. Yuan.
Binomial Difference Ideals,
{\em Journal of Symbolic Computation},
DOI: 10.1016/j.jsc.2016.07.029, 2016.

\bibitem{gao-tdv}
X.~S. Gao, Z. Huang, J. Wang, C.~M. Yuan.
Toric Difference Variety,
accepted by Journal of Systems Science and Complexity, 2016.

\bibitem{Hrushovski1}
E. Hrushovski.
The Elementary Theory of the Frobenius Automorphisms,
Available from http://www.ma.huji.ac.il/\~\,ehud/, July, 2012.


\bibitem{levin}
A. Levin.
{\em Difference Algebra},
Springer-Verlag, New Work, 2008.

\bibitem{levin1}
A. Levin.
On the Ascending Chain Condition for Mixed Difference Ideals,
{\em Int. Math. Res. Not.}, 2015(10), 2830-2840,
DOI: 10.1093/imrn/rnu021, 2015.



\bibitem{Jie}
J. Wang.
Monomial Difference Ideals,
{\em Proc. Amer. Math. Soc.}, 
DOI: 10.1090/proc/13369, 2016.

\bibitem{wibmer}
M. Wibmer.
{\em Algebraic Difference Equations},
lecture notes, 2013.

\bibitem{wibmer-group}
M. Wibmer.
{\em Affine difference algebraic groups},
arXiv:1405.6603, 2014.

\end{thebibliography}
\end{document}